\documentclass[12pt]{amsart}
\usepackage[a4paper]{geometry}
\usepackage{color}

 \newtheorem{thm}{Theorem}[section]
 \newtheorem{cor}[thm]{Corollary}
 \newtheorem{lem}[thm]{Lemma}
 
 \theoremstyle{definition}
 \newtheorem{defn}{Definition}[section]
 \theoremstyle{remark}
 
 \numberwithin{equation}{section}
\newcommand{\grad}{\mathop{\mathrm{grad}}}
\newcommand{\trace}{\mathop{\mathrm{trace}}}
\newcommand{\dive}{\mathop{\mathrm{div}}}

\begin{document}
\title{\textbf{Biconservative hypersurfaces in space forms with  harmonic curvature}}
\author{Ram Shankar Gupta, Andreas Arvanitoyeorgos and Marina Statha}
 \maketitle
\begin{abstract}
A biconservative hypersurface $M$
in a space form $\overline M^{n+1}(c)$ is a hypersurface
whose normal part of the vector $\Delta\vec{H}$ vanishes.
We study such hypersurfaces with harmonic curvature or parallel Ricci curvature, and   prove
that have constant mean
curvature.   
Also, by combining our results with previous results of H.B. Lawson Jr. and M. Umehara we obtain that
if $M$ is a biconservative hypersurface in
  a space form $\overline M^{n+1}(c)$ with parallel Ricci tensor or harmonic curvature, then the shape operator $\bold{A}$ of $M$  is covariant constant, i.e.
  $\nabla\bold{A}=0$.
\end{abstract}

\maketitle
\section{\textbf{Introduction}}

A Riemannian manifold $(M,g)$ is said to have {\it harmonic curvature}
if its Levi-Civita connection $\nabla$ in the tangent bundle $TM$ satisfies
$\delta R = 0$, where $R$ is the curvature tensor and  $\delta$ is the codifferential of the formal adjoint of the  exterior differentiation operator $d$. The Ricci tensor $S $ can be regarded as a $1$-form with values in the cotangent bundle.  Then from the contraction of the second Bianchi identity, the notion of harmonic curvature is equivalent to $ dS = 0$ (\cite{h32}).  The curvatures
of such manifolds occur as a special case of \emph{Yang-Mills fields} when $M$ is compact.  This means that
$\nabla$ is a Yang-Mills potential in $TM$, that is, a critical point for the \emph{Yang-Mills functional}
\begin{eqnarray}\textsf{Y} M(\nabla) = \frac{1}{2}\int_M |R|^2,\end{eqnarray}
in the space of all $g$-metric connections.
 These
manifolds also form a natural generalization of Einstein manifolds and
of conformally flat manifolds with constant scalar curvature (\cite{h32}).

In a preliminary version of \cite{h6} J. P. Bourguignon conjectured that on a $4$-dimensional compact Riemannian manifold with harmonic curvature the Ricci tensor must be {\it parallel}, that is $\nabla S=0$; equivalently $(\nabla _X S)(Y, Z)=0$, for all $X, Y, Z\in TM$.  The question was negatively answered by Derdzi$\acute{\textrm{n}}$ski \cite{h2} who gave an example of a $4$-dimentional Riemannian manifold with harmonic curvature whose Ricci tensor is not parallel. In fact, the harmonic curvature condition on a manifold is essentially weaker than that of  parallel Ricci tensor (\cite{h3}, \cite{h5}).  On the other hand, hypersurfaces with parallel Ricci tensor in a Riemannian manifold of constant curvature were studied by H. B. Lawson Jr. \cite{h7} and I. Mogi, H. Nakagawa \cite{h8}.

E. $\hat{\textrm{O}}$machi in \cite{h4} investigated compact hypersurfaces with harmonic curvature in a Euclidean space or a sphere and gave a classification of such hypersurfaces provided that the mean curvature is constant.
M. Umehara in \cite{h1} studied hypersurfaces with harmonic curvature immersed in a Riemannian manifold of constant curvature, in which $\hat{\textrm{O}}$machi's result   was generalized without the assumption of compactness.
Also, U-H. Ki, H. Nakagawa, M. Umehara in \cite{h9} classified complete hypersurfaces with harmonic curvature in a Riemannian manifold of constant curvature, in the case where the multiplicity of each principal curvature is greater than one.  They proved that there exist infinitely many hypersurfaces with harmonic curvature and non parallel Ricci tensor.
Finally, certain kinds of Riemannian manifolds with harmonic curvature have been investigated by J. P. Bourguignon \cite{h6}, A. Derdzinski \cite{h3}, T. Kashiwada \cite{h10}, S. Tachibana \cite{h11} and others.

At this point we would like to recall the work \cite{h5} of A. Gray, in which he introduced certain generalizations of  the Einstein condition in a Riemannian manifold $(M,g)$.
He called a Riemannian metric $g$ of $\mathcal{A}$-type if $(\nabla _X S)(X, X)=0$ for all $X\in TM$, and of  $\mathcal{B}$-type if $(\nabla _X S)(Y, Z)=(\nabla _Y S)(X, Z)$ for all $X, Y, Z\in TM$, that is Ricci tensor is a {\it Codazzi tensor}.
Manifolds that satisfy either one of these conditions are called {\it Einstein-like} manifolds.
 One of the main results in \cite{h5} is that a Riemannian manifold $(M,g)$ has parallel Ricci tensor if and only if $g$ is  of $\mathcal{A}$-type {\it and} of $\mathcal{B}$-type.

On the other hand there has been extensive research interest in
 the study of biharmonic submanifolds and biconservative submanifolds.
They are generalizations of harmonic maps proposed in 1964 by Eells and
Sampson \cite{7}, which are critical points of the bienergy
functional, obtained by integrating the squared norm of the tension
field, and are characterized by the vanishing of the bitension
field. Another interesting research direction, derived from here, is
the study of {\it biconservative} submanifolds, i.e., those submanifolds
for which only the tangent part of the bitension field vanishes. 
The terminology was introduced by R. Caddeo et. al. in \cite{2}.
For updated works on
biharmonic and biconservative submanifolds, we refer to 
 \cite{ch}, \cite{5},  \cite{9}, \cite{10}, \cite{11}, \cite{14}, \cite{16}, \cite{19},
 \cite{mo}, \cite{8}, \cite{oc},   
   \cite{22} and references therein.

\smallskip
 In view of the above developments, in the present paper we study biconservative
hypersurfaces in a space form $\overline{M}^{n+1}(c)$  with harmonic curvature. 

 In the context of the paper's results, a prerequisite is having a smooth suitable local tangent orthonormal frame denoted by $\{e_1, \dots, e_n \}$  with corresponding smooth eigenvalues of shape operator. It is important to note that the smoothness of the eigenvalue functions, is limited to an open dense subset where the multiplicities of eigenvalues are locally constant. However, extending this conclusion to claim that these multiplicities remain consistent across all connected components within this open dense set, likely requires referencing the findings of DeTurck and Goldschmidt in \cite{DH}. Specifically, they established that in appropriately chosen local coordinates, any metric with $\dive R=0$, i.e. divergence equal to zero for the curvature tensor $R$, is real analytic.

The main results of the paper are  the following:

\begin{thm}\label{main1}
Let $M$ be a biconservative hypersurface with $(\nabla_{X}S)(X, X) = 0$ for all $X\in TM$ in
  a space form $\overline M^{n+1}(c)$. Then $M$  has constant mean curvature.
 \end{thm}

As a consequence we obtain the following:

\begin{cor}\label{cor1}
Let $M$ be a biconservative hypersurface with parallel Ricci tensor in
  a space form $\overline M^{n+1}(c)$. Then $M$  has constant mean curvature.
\end{cor}

In addition, by taking into consideration a result of H.B. Lawson in  \cite[Proposition 2]{h7}  (cf. also \cite[Proposition 5]{ry}) we obtain the following:
\begin{cor}
Let $M$ be a biconservative hypersurface with parallel Ricci tensor in
  a space form $\overline M^{n+1}(c)$. Then the shape operator $\bold{A}$ of $M$  is covariant constant, i.e.
  $\nabla\bold{A}=0$.
\end{cor}

\noindent
For the case of harmonic curvature we prove the following:

\begin{thm}\label{main2}
Every biconservative  hypersurface $M$ in a space form
$\overline{M}^{n+1}(c)$ with harmonic curvature, is of constant mean
curvature.
 \end{thm}

By taking into consideration a result of M. Umehara (\cite[Theorem 2.1]{h1}) that states that, if a hypersurface in a space form has constant mean curvature and  harmonic curvature, then it has
either one or two distinct constant principal curvatures, we obtain the following:

\begin{cor} Every biconservative  hypersurface $M$ in a space form
$\overline{M}^{n+1}(c)$ with harmonic curvature has $\nabla\bold{A}=0$.
\end{cor}

 The class of biconservative hypersurfaces in space forms  is contained in some  other classes of  hypersurfaces, such as biharmonic   hypersurfaces
or hypersurfaces satisfying  $\triangle \vec
 {H}=\lambda \vec{H}$.  Therefore,  from Theorem \ref{main2} we have following:

\smallskip
\begin{cor}
Every biharmonic  hypersurface $M$ in a space form
$\overline{M}^{n+1}(c)$ with harmonic curvature, is of constant mean
curvature. \end{cor}
 \begin{cor}
Every hypersurface $M$ in a space form
$\overline{M}^{n+1}(c)$ satisfying $\triangle \vec
 {H}=\lambda \vec{H}$ with harmonic curvature, is of constant mean
curvature.
 \end{cor}



 The paper is organized as follows.  In Section 2 we give some preliminaries.  In Section 3 we prove some properties of biconservative hypersurfaces in space forms.  In Section 4 we use the condition of harmonic curvature to prove Lemma  \ref{lemm} and Theorem \ref{main2}.
 Finally, in Section 5 we provide some counterexamples and make some remarks.

 \medskip
 \noindent
{\bf Acknowledgments.} The first author acknowledges the support of the award of a grant under FRGS for the year 2024-25, F.No. GGSIPU/DRC/2024/2437/7. The third author was partially supported by a grant from the Empirikion Foundation in Athens.
All authors would like to thank the referee for raising various questions which resulted in the
inclusion of the last section of the paper.

\section{\textbf{Preliminaries}}

  Let ($M^n, g$) be a hypersurface isometrically immersed in a $(n+1)$-dimensional space form
$(\overline{M}^{n+1}(c), \overline g)$ and $g = \overline g_{|M}$ be the induced metric.
We denote by $\xi$ a unit normal vector to $M$ where
$\overline{g}(\xi, \xi)= 1$, and by $h$ and $ \bold{A}$ the second fundamental form and the shape operator of $M$, where $h(X, Y)=g(\bold{A}(X), Y)$, for $ X, Y \in \Gamma(TM)$.

 The mean curvature $H$ of $M$ is given by
\begin{equation}\label{e:b4}
H = \frac{1}{n} \trace \bold{A},
\end{equation}
and the mean curvature vector $\vec{H}=H\xi$ is a well defined normal vector field to $M$ in $\overline{M}$.

 Let $\nabla$ denote the Levi-Civita connection of $M$. Then, the Gauss and Codazzi equations are given by
\begin{equation}\label{e:b5}
R(X, Y)Z = c(g(Y, Z) X - g(X, Z) Y)+g(\bold{A}Y, Z) \bold{A}X - g(\bold{A}X, Z) \bold{A}Y,
\end{equation}
\begin{equation}\label{e:b6}
(\nabla_{X}\bold{A})Y = (\nabla_{Y}\bold{A})X
\end{equation}
respectively, where $R$ is the curvature tensor and
\begin{equation}\label{e:b7}
(\nabla_{X}\bold{A})Y = \nabla_{X}\bold{A} Y-
\bold{A}(\nabla_{X}Y)
\end{equation}
for all $ X, Y, Z \in \Gamma(TM)$.

The Ricci curvature $S$ on $M$ is defined as
\begin{eqnarray}
S(X, Y) = \sum_{i = 1}^{n} g(R(e_i, X)Y, e_i), \label{ri1}
\end{eqnarray}
with respect to a suitable orthonormal frame field $\{e_i\}_{i=1}^{n}$,   $X, Y$ $\in$ $TM$.
 The \emph{Ricci operator} is defined by the equation
 $S(X,Y)=g(QX,Y)$.

 An immersion \(\phi : (M^n, g) \rightarrow (N^m, h)\) is called biharmonic if and only if its mean curvature vector field \(\vec{H}\) satisfies the equation (\cite{8})
\begin{equation}
\Delta \vec{H} + \operatorname{trace} R^N (d\phi, \vec{H}) d\phi = 0, \label{bis1}
\end{equation}
where \((M^n, g)\) and \((N^m, h)\) are Riemannian manifolds of dimensions \(n\) and \(m\) respectively, and \(R^N\) is the curvature tensor of \(N\).

The biharmonic
condition (\ref{bis1})  for $M^n$  in  a $(n+1)$-dimensional space form $\overline
M^{n+1}(c)$  can be decomposed into its normal and tangent part  as follows (\cite{3}, \cite{8}):
\begin{eqnarray}
2\bold{A} (\grad H)+ n  H \grad H = 0,\label{e:b9}\\
 \Delta H+H({\rm trace}(\bold{A}^2)-n c)=0.\label{bis2}
\end{eqnarray}
\begin{defn}
A hypersurface in space form  satisfying equation (\ref{e:b9})
 is called biconservative.
\end{defn}
In the present work we are concerned with biconservative
hypersurfaces $M^n$  in  a $(n+1)$-dimensional space form $\overline
M^{n+1}(c)$.

\smallskip
The following algebraic lemma will be useful to get our result:

\begin{lem} {\rm (\cite[Theorem 4.4, pp. 58--59]{ke})} Let D be a unique factorization domain, and let $f(X) =
a_{0}X^{m} +a_{1}X^{m-1} + \dots + a_{m}, g(X) = b_{0}X^{n} +
b_{1}X^{n-1} + \dots + b_{n}$ be two polynomials in $D[X]$. Assume
that the leading coefficients $a_{0}$ and $b_{0}$ of $f(X)$ and
$g(X)$ are not both zero. Then $f(X)$ and $g(X)$ have a nonconstant
common factor iff the resultant $\Re(f, g)$ of $f$ and $g$ is zero,
where
\begin{center}
$\Re(f,g)=
\begin{vmatrix}
  a_{0} & a_{1} & a_{2} & \cdots & a_{m} &   &   &   \\
    & a_{0} & a_{1} & \cdots & \cdots & a_{m} &   &   \\
    &   & \ddots & \ddots & \ddots & \ddots & \ddots &   \\
   &   &   & a_{0} & a_{1} & a_{2} & \cdots & a_{m} \\
  b_{0} & b_{1} & b_{2} & \cdots & b_{n} &   &   &   \\
    & b_{0} & b_{1} & \cdots & \cdots & b_{n} &  &   \\
    &   & \ddots & \ddots & \ddots & \ddots & \ddots &   \\
    &   &   & b_{0} & b_{1} & b_{2} & \cdots & b_{n} \\
\end{vmatrix},
$
\end{center}
and there are n rows of ``a" entries and m rows of ``b" entries.
\end{lem}


\section{\textbf{Biconservative hypersurfaces in space forms}}

In this section, we study biconservative hypersurfaces  in space forms. In view of
(\ref{e:b9}), it is obvious  that any constant mean curvature (CMC) hypersurface is
biconservative.
Therefore, we are interested in the study of
non CMC biconservative hypersurfaces in a space form $\overline
M^{n+1}(c)$, so we assume that the mean curvature is not
constant.
This implies that $\grad H\neq0$, hence there exists an
 open connected subset $U$ of $M^n$ with
${\rm grad}_{x}H \neq 0$, for all $x\in U$. From (\ref{e:b9}), it is  evident 
that $\grad H$ is an eigenvector of the shape operator
$\bold{A}$ with the corresponding principal curvature
$-\frac{nH}{2}$.

We introduce the following sets
$$ A=\{1, 2,\dots,n\},\quad  B=\{2, 3,\dots, n\}.$$

Without losing generality, we choose $e_{1}$ in the direction of
$\grad H$. Then, the shape operator $\bold{A}$ of a hypersurface
$M^n$ in a space form $\overline M^{n+1}(c)$ takes the following
form with respect to a suitable orthonormal frame $\{e_{1},
e_{2},\dots, e_{n}\}$,
\begin{equation}\label{e:c1} \bold{A}e_i=\lambda_ie_i, \quad i \in
A,
\end{equation}
where $\lambda_i$ is the eigenvalue corresponding to the eigenvector $e_i$
of the shape operator.

 Then $\grad H$ can be expressed as
\begin{equation}\label{e:c2}
\grad H =\sum_{i=1}^{n} e_{i}(H)e_{i}.
\end{equation}

 As we have taken $e_{1}$ parallel to $\grad H$, it follows that
\begin{equation}\label{e:c3}
e_{1}(H)\neq 0, \quad e_{i}(H)= 0, \quad i \in B.
\end{equation}

  We express
\begin{equation}\label{e:c4}
\nabla_{e_{i}}e_{j}=\sum_{m=1}^{n}\omega_{ij}^{m}e_{m}, \quad i, j
\in A.
\end{equation}

From (\ref{e:c4}) and using the compatibility conditions
$(\nabla_{e_{k}}g)(e_{i}, e_{i})= 0$ and $(\nabla_{e_{k}}g)(e_{i},
e_{j})= 0$, we obtain
\begin{equation}\label{e:c5}
\omega_{ki}^{i}=0, \hspace{1 cm} \omega_{ki}^{j}+ \omega_{kj}^{i}
=0,
\end{equation}
for $i \neq j$ and $i, j, k \in A$.

Taking $X=e_{i}, Y=e_{j}$ in (\ref{e:b7}) and using (\ref{e:c1}),
(\ref{e:c4}), we get

\begin{eqnarray}\label{e:A1}(\nabla_{e_{i}}\bold{A})e_{j}=e_{i}(\lambda_{j})e_{j}+\sum_{k=1}^{n}
\omega_{ij}^{k}(\lambda_{j}-\lambda_{k})e_{k}.
\end{eqnarray}

Using (\ref{e:A1}) in
(\ref{e:b6}), we find

$$e_{i}(\lambda_{j})e_{j}+\sum_{k=1}^{n}\omega_{ij}^{k}
(\lambda_{j}-\lambda_{k})e_k=e_{j}(\lambda_{i})e_{i}+\sum_{k=1}^{n}\omega_{ji}^{k}
(\lambda_{i}-\lambda_{k})e_k.$$

Taking the inner product in the above equation in a convenient way,
we
obtain

\begin{equation}\label{e:c6}
e_{i}(\lambda_{j})=
(\lambda_{i}-\lambda_{j})\omega_{ji}^{j}=(\lambda_{j}-\lambda_{i})\omega_{jj}^{i},
\end{equation}
and
\begin{equation}\label{e:c7}
(\lambda_{i}-\lambda_{j})\omega_{ki}^{j}=
(\lambda_{k}-\lambda_{j})\omega_{ik}^{j},
\end{equation}
 for $i\ne j\ne k\ne i$.

\medskip
Next, we will need the following two lemmas, whose proofs or variants can be found in several
references (e.g. \cite[Lemmas 3.1, 3.2]{fh}, \cite[Lemma 4.1]{22}).

\begin{lem}\label{l:1}
 Let $M^n$ be a biconservative hypersurface  in
  a space form $\overline M^{n+1}(c)$ having the shape operator given by
$(\ref{e:c1})$ with respect to a suitable orthonormal frame
$\{e_{1}, e_{2},\dots, e_{n}\}$. Then
\begin{equation}\label{e:k4}
\lambda_1\neq\lambda_j,\ \ \mbox{for all}
\ \ j\in B.
\end{equation}
\end{lem}


\medskip
 In view of the Lemma \ref{l:1},
$\lambda_1= -\frac{n}{2} H$ has multiplicity one.
Since $M$ can have at most $n$ distinct principal curvatures, we  assume that these are
$\mu_1$, $\mu_{2}, \mu_{3}, \dots, \mu_{q}$, $q\leq n-1$,
 with
multiplicities $1$, $p_2, p_3, \dots, p_q$ respectively,
such that
$$\lambda_2=\lambda_3=\dots=\lambda_{p_{2}+1}\equiv\mu_2,$$
$$\lambda_{p_2+2}=\lambda_{p_2+3}=\dots=\lambda_{p_2+p_3+1}\equiv\mu_3,$$
$$\lambda_{p_2+p_3+2}=\lambda_{p_2+p_3+3}=\dots=\lambda_{p_2+p_3+p_4+1}\equiv\mu_4,$$
$$ \cdots$$
$$ \lambda_{p_2+p_3+\cdots +p_{q-1}+2}=\dots=\lambda_{p_2+p_3+\cdots +p_q+1}\equiv\mu_q.$$
Here $p_2+p_3+\dots + p_q+1=n$.

 Using (\ref{e:b4}) and (\ref{e:c1}), we obtain that
 \begin{equation}\label{e:c11}
\sum_{i=2}^n\lambda_i=p_2\mu_2+p_3\mu_3+\cdots+p_q\mu_q=\frac{3n}2H=-3\lambda_1.
\end{equation}

We introduce the sets
\begin{eqnarray*}
C_1&=&\{2, 3,\dots, p_2+1\}\\
C_2&=&\{p_2+2, p_2+3,\dots, p_2+p_3+1\}\\
C_3&=&\{p_2+p_3+2, p_2+p_3+3,\dots, p_2+p_3+p_4+1\}\\
&\vdots&\\
 C_{q-1}&=&\{p_2+p_3+\cdots +p_{q-1}+2,\dots , p_2+p_3+\cdots +p_{q}+1=n\}.
\end{eqnarray*}

\medskip
We also need the following:

\begin{lem}\label{l:2}
 Let $M^n$ be a biconservative hypersurface in
  a space form $\overline M^{n+1}(c)$ having the shape operator given by
$(\ref{e:c1})$ with respect to a suitable orthonormal frame
$\{e_{1}, e_{2},\dots, e_{n}\}$. Then the following relations are satisfied:

\begin{eqnarray} \nabla_{e_{1}}e_{1}=0, \quad \nabla_{e_{1}}e_{i}=\sum_{m\in C_{s}}\omega_{1i}^me_{m}\ \ \mbox{for all}\ \  i \in
C_{s}, \   s\in\{1,2,\dots,q-1\},\nonumber\end{eqnarray}
 \begin{eqnarray}
\nabla_{e_{i}}e_{1}= -\omega_{ii}^{1}e_{i} \ \ \mbox{for all}\ \  i \in B, \quad \nabla_{e_{i}}e_{i}=\sum_{m=1}^n\omega_{ii}^{m}e_{m}
\ \ \mbox{for all}\ \  i \in B,\nonumber \end{eqnarray}

\begin{eqnarray}\nabla_{e_{i}}e_{j}=\sum_{m\in C_s}\omega_{ij}^{m}e_{m} \ \ \mbox{for all}\ \  i, j \in C_s, \  s\in\{1,2,\dots,q-1\}, \quad i\neq j, \quad j\neq m, \nonumber \end{eqnarray}
\begin{eqnarray}
\nabla_{e_{i}}e_{j}=\omega_{ij}^i e_i+\sum_{m\in B\setminus C_s}\omega_{ij}^{m}e_{m} \ \ \mbox{for all}\ \  i \in C_s, \quad j \in
B\setminus C_s,  s\in\{1,2,\dots,q-1\},  \quad m\neq j,\nonumber \end{eqnarray}

where $\sum_{C_s}$ and $\sum_{B\setminus C_s}$ denote
the summation taken over the corresponding $C_s$ and $B\setminus C_s $, respectively  for  $s\in\{1,2,\dots,q-1\}$,  and $\omega_{ij}^{i}$ 
satisfy $(\ref{e:c5})$ and $(\ref{e:c6})$.\\
\end{lem}

By using Lemma \ref{l:2}, (\ref{e:b5}) and (\ref{e:c1}) we evaluate $g(R(e_{1},e_{i})e_{1},e_{i})$,  and we find the
following relation:
\begin{equation}\label{e:c14}
 e_{1}(\omega_{ii}^{1})- (\omega_{ii}^{1})^{2}= c+\lambda_{1} \lambda_{i}, \quad i \in
 B.
\end{equation}
Also, using (\ref{e:c3}), Lemma \ref{l:2}, and the fact that $[e_{i}\hspace{.1 cm}e_{1}](H)=0= \nabla_{e_{i}}e_{1}(H)-\nabla_{e_{1}}e_{i}(H),$  we find that
\begin{equation}\label{e:c17}
 e_{i}e_{1}(H)= 0, \ \ \mbox{for all}\ \  i\in B.
\end{equation}

%
%
%
%
%

\section{\textbf{Biconservative hypersurfaces in space forms with harmonic curvature}}

In this section we study biconservative hypersurfaces $M^n$ in
space forms $\overline M^{n+1}(c)$ with some condition on Ricci tensor.

 Every
Riemannian manifold with parallel Ricci tensor has harmonic curvature.
As discussed in the Introduction the converse statement is not necessarily true for compact manifolds.
Then  a natural question arises whether manifolds with weaker conditions than parallel Ricci tensor could exist or not. Every parallel tensor in a Riemannian manifold is a Coddazi tensor, which occurs in a natural way in the study of a Riemannian manifold with harmonic curvature.

Therefore, in this section first we study biconservative hypersurfaces $M^n$ in
space forms $\overline M^{n+1}(c)$ with parallel Ricci tensors and later on biconservative hypersurfaces $M^n$ with harmonic curvature.

A hypersurface is said to have {\it parallel Ricci tensor} $S$ if
\begin{eqnarray}
(\nabla_{Z}S)(X, Y) = g((\nabla_{Z}Q)X, Y)=0, \label{ri2}
\end{eqnarray}
 for all $X, Y, Z$ $\in$ $TM$.

\smallskip

 We can now prove the first main result:\\

\noindent
\emph{\textbf{Proof of Theorem 1.1.}} Using (\ref{ri1}), (\ref{e:b4}) and (\ref{e:b5}), we obtain that
\begin{equation}\label{e:d1}
 S(X,Y)=(n-1) c\, g(X,Y)+nHg(\bold{A}X,Y)-g(\bold{A}X,\bold{A}Y),
\end{equation}
for all $X, Y$ $\in$ $TM$.

Now, taking $X=Y=e_1$ in (\ref{e:d1}) and using (\ref{e:c1}) and $\lambda_1=-\frac{nH}{2}$,
we find
\begin{equation}\label{e:d3}
S(e_1,e_1)=(n-1) c-\frac{3n^2 H^2}{4}.
\end{equation}

On the other hand, we have
\begin{equation}\label{e:d4}
\nabla_{e_1}S(e_1,e_1)= (\nabla_{e_1}S)(e_1,e_1)+2S(\nabla_{e_1}e_1,e_1).
\end{equation}

 Since $g(\nabla_{e_1}e_1,e_1)=0$, therefore from  (\ref{e:d1}) and (\ref{e:c1}), we have
 \begin{equation}\label{e:d5}
 S(\nabla_{e_1}e_1,e_1)=0.
\end{equation}

\smallskip
 Now, taking covariant derivative of  (\ref{e:d3}) with respect to $e_1$ and using
 (\ref{e:d4}), (\ref{e:d5})  and the condition of $\mathcal{A}$-type metric $(\nabla_XS)(X, X)=0$  we get that 
\begin{equation}\label{e:d6}
H \, e_1(H)=0,
\end{equation}
which gives a contradiction to the fact that $e_1(H)\neq 0$, hence $H$ must be constant.  
 Thus the proof is complete. \hfill $\Box$

\medskip
If we replace the condition of $\mathcal{A}$-type metric by the  stronger   condition of parallel Ricci tensor, we obtain Corollary 1.2.   
Note that since $M$ has parallel Ricci tensor,  taking the trace of (\ref{ri2}) we find that the scalar curvature is also constant.

\medskip
Next we consider  biconservative hypersurfaces with harmonic curvature.  We need the following:

  \begin{lem}\label{lemm}
Let $M$ be a biconservative hypersurface with harmonic curvature in
  a space form $\overline M^{n+1}(c)$. Then,
  \begin{eqnarray}\label{e:d12}
 n \, e_1(H) \lambda_i= (\lambda_{i}-\lambda_{1})^2 \, \omega_{ii}^{1},\qquad
 \mbox{for all}\ \  i\in B
\end{eqnarray}
\begin{eqnarray}\label{e:d15}
 e_i(\lambda_j)=0, \ \mbox{for}\  \, i\neq j  \in  B,
\end{eqnarray}
\begin{eqnarray}\label{e:d15a}
 e_i(\lambda_i)=0,
\end{eqnarray}
for $i\in B$.
  \end{lem}

  \begin{proof} Suppose that $M$ has harmonic curvature. Then

  \begin{eqnarray}\label{e:d6}
(\nabla_{X}Q)Y = (\nabla_{Y}Q)X,
\end{eqnarray}
for $X, Y\in TM$ (cf. by \cite{A1}, \cite{h1}).

 From (\ref{e:d1}) we obtain that
 \begin{equation}\label{e:d7}
 QY=(n-1) c\, Y+n\,H \bold{A}Y-\bold{A}^2Y.
\end{equation}

Taking covariant derivative of (\ref{e:d7}) with respect to $X$ and using (\ref{e:d7}) and (\ref{e:b7}), we find
\begin{eqnarray}\label{e:d8}
(\nabla_{X}Q)Y = n\, H(\nabla_{X}\bold{A})Y+n \, (XH) \bold{A}Y-(\nabla_{X}\bold{A})\bold{A}Y-\bold{A}(\nabla_{X}\bold{A})Y.
\end{eqnarray}

  Using (\ref{e:d8}) and (\ref{e:b6}) in (\ref{e:d6}), we obtain
\begin{eqnarray}\label{e:d9}
 n \, (XH) \bold{A}Y-(\nabla_{X}\bold{A})\bold{A}Y=n \, (YH) \bold{A}X-(\nabla_{Y}\bold{A})\bold{A}X.
\end{eqnarray}

Putting $X=e_1$ and $Y=e_i$ in (\ref{e:d9}), and using (\ref{e:c1}) and (\ref{e:c3}), we find
\begin{eqnarray}\label{e:d10}
 n \, e_1(H) \lambda_i \, e_i=(\lambda_i- \lambda_1) \, (\nabla_{e_i}\bold{A})e_1,
\end{eqnarray}
for $i\in B$.

Taking $j=1$ in (\ref{e:A1}) and using (\ref{e:c5}) and (\ref{e:c3}), we have

\begin{eqnarray}\label{e:d11}
 (\nabla_{e_i}\bold{A})e_1= \sum_{k=1}^{n}\omega_{ik}^{1}
(\lambda_{k}-\lambda_{1})e_k.
\end{eqnarray}

Using (\ref{e:d11}) in (\ref{e:d10}) and comparing the coefficient of $e_i$ on both sides, we obtain (\ref{e:d12}).

Now, putting $X=e_i$ and $Y=e_j$, $i\neq j  \in  B$ in (\ref{e:d9}), and using (\ref{e:c1}) and (\ref{e:c3}), we find that
\begin{eqnarray}\label{e:d13}
 (\lambda_j- \lambda_i) \, (\nabla_{e_i}\bold{A})e_j=0.
\end{eqnarray}

If $\lambda_i\ne \lambda _j$
from (\ref{e:d13}), we get that 
\begin{eqnarray}\label{e:d14}
 (\nabla_{e_i}\bold{A})e_j=0,
\end{eqnarray}
due to more than two distinct principal curvatures.
Using (\ref{e:d14}) in (\ref{e:A1}) and comparing the coefficient of $e_j$, we obtain (\ref{e:d15}).

If $\lambda_i=\lambda_j$ and $i\ne j$, then by a result in \cite{cr} we obtain (\ref{e:d15}).

Finally, differentiating (\ref{e:c11}) with respect to $e_i$ for $i\in B$, and using (\ref{e:d15}) and (\ref{e:c3}), we obtain (\ref{e:d15a}). Thus the proof is complete.
\smallskip
\end{proof}

In view of (\ref{e:c3}) and (\ref{e:d15a}), we conclude that all the principal curvatures $\lambda_i$ depend on a single parameter $t$. Thus, we consider an integral curve of $e_1=\frac{d}{dt}$ passing through $x = \gamma(t_0)$ as
$\gamma(t), t \in I$. It is easy to show that there exists a local chart $(U; t = x_1, x_2,\dots,x_n)$ around $x$, such that $\lambda_i(t, x_2,\dots,x_n) = \lambda_i(t)$,  $i=1,\dots,n$ in the entire neighbourhood of $x$.
We can now  prove Theorem \ref{main2}:\\

\noindent
  \emph{\textbf{Proof of Theorem \ref{main2}.}} From (\ref{e:c6}) we get
  \begin{eqnarray}\label{e:d32}
    e_1(\lambda_i)=(\lambda _i-\lambda _1) \omega^1 _{ii},\quad i\in B.
  \end{eqnarray}

   Eliminating $\omega^1 _{ii}$ from (\ref{e:d32}) using (\ref{e:d12}), we find that
  \begin{eqnarray}\label{e:d33}
   n \, \lambda _i \,  e_1(H)=\left(\lambda _i-\lambda _1\right) e_1(\lambda_i).
  \end{eqnarray}

 Using $\lambda_1=-\frac{n\, H}{2}$ in  (\ref{e:d32}), we obtain
  \begin{eqnarray}\label{e:d34}
-2 \, \lambda _i \,  e_1(\lambda_1)=\left(\lambda _i-\lambda _1\right) e_1(\lambda_i),
  \end{eqnarray}
which can be written as
 \begin{eqnarray}\label{e:d35}
  -2 \, \lambda _i \,  \frac{d\lambda_1}{dt}=\left(\lambda _i-\lambda _1\right) \frac{d\lambda_i}{dt},
 \end{eqnarray}
  and this  gives
  \begin{eqnarray}\label{e:d36}
    \frac{d\lambda_1}{d\lambda_i}-\frac{\lambda_1}{2\, \lambda_i}= -\frac{1}{2}.
  \end{eqnarray}

  Equation  (\ref{e:d36}) is a $Leibnitz$ linear differential equation and has solution
 \begin{eqnarray}\label{e:d37}
\lambda_1+\lambda_i=c_i \, \sqrt{|{\lambda_i}|},
  \end{eqnarray}
  where $c_i$ is a constant. 

  Solving (\ref{e:d37}) for $\lambda_i$ we find
  \begin{eqnarray}\label{e:d38}
  \lambda_i= \frac{1}{2} \left(-2 \lambda_1+c_i^2\pm \sqrt{c_i^4\pm 4 \lambda_1 c_i^2} \right).
  \end{eqnarray}
\smallskip
In view of (\ref{e:d38}) and (\ref{e:d34}), we conclude that \(\lambda_i \neq 0\), \(\lambda_i \pm \lambda_1 \neq 0\), and \(\omega^1_{ii} \neq 0\). Otherwise, we arrive at a contradiction implying that \(\lambda_1\) is constant.

Now, we consider following two cases:

\smallskip
\noindent
\textbf{Case I: } If $\lambda_2=\lambda_3=\dots=\lambda_{n}$, then using this and
 (\ref{e:d38}), we get  

\begin{eqnarray}\label{e:d38c}
  \lambda_2=\lambda_3=\dots=\lambda_{n}= \frac{1}{2} \left(-2 \lambda_1+c_2^2\pm \sqrt{c_2^4\pm 4 \lambda_1 c_2^2} \right).
  \end{eqnarray}

Now, using (\ref{e:d38c}) in (\ref{e:c11}) we obtain
\begin{eqnarray}\label{e:d39}
  -3\, \lambda_1= \frac{(n-1)}{2} \left(-2 \lambda_1+c_2^2\pm \sqrt{c_2^4\pm4 \lambda_1 c_2^2} \right),
  \end{eqnarray}
 which gives
\begin{eqnarray}\label{e:d40}
 \lambda_1 \left(\lambda_1 (n-4)^2-(n-4)(n-1) c_i^2 \mp (n-1)^2 c_i^2\right)=0.
  \end{eqnarray}
  
    From (\ref{e:d40}) we get that $\lambda_1$ is a constant, a contradiction to (\ref{e:c3}). 
   
   \smallskip
\noindent
 \textbf{Case II: } If $\lambda_i\neq\lambda_j$, for some $i \neq j \in B$. Then using (\ref{e:d34}) and (\ref{e:d32}), we get
   \begin{eqnarray}\label{e:d341}
\frac{\left(\lambda _i-\lambda _1\right)^2 \omega^1_{ii}}{\lambda_i}=\frac{\left(\lambda _j-\lambda _1\right)^2 \omega^1_{jj}}{\lambda_j},
  \end{eqnarray}
  which gives
   \begin{eqnarray}\label{e:d342}
\left(\lambda _i-\lambda _1\right)^2 \omega^1_{ii} \lambda_j-\left(\lambda _j-\lambda _1\right)^2 \omega^1_{jj}\lambda_i=0.
  \end{eqnarray}
  
  Differentiating  (\ref{e:d342}) with respect to $e_1$ and using (\ref{e:d32}) and (\ref{e:c14}), we obtain
 \begin{eqnarray}\label{e:d343}
b_1  \, e_1(\lambda_1) + (c+\lambda _1 \lambda _i+(\omega^1_{ii})^2) \, b_2- \left(c+\lambda _1 \lambda _j+(\omega^1 _{jj})^2\right) \, b_3 \\ \nonumber+\left(\lambda _j-\lambda _1\right) \omega^1 _{\text{jj}} \, b_4+\left(\lambda _i-\lambda _1\right) \omega^1_{\text{ii}} \, b_5=0,
  \end{eqnarray}
  where
    \begin{align*}
   b_1 &= 2 (\lambda _1-\lambda _i) \omega^1_{ii} \lambda _j+2 \lambda _i (\lambda _j-\lambda _1) \omega^1_{jj}, \, \, \,
   b_2 = (\lambda _i-\lambda _1)^2 \lambda _j, \, \, \,
   b_3 =  \left(\lambda _j-\lambda _1\right)^2 \lambda _i, \\
   b_4 &= \left(\lambda _1-\lambda _i\right){}^2 \omega^1 _{\text{ii}}+2 \lambda _i \left(\lambda _1-\lambda _j\right) \omega^1 _{\text{jj}}, \, \, \, \,\, \,
   b_5 = 2 \left(\lambda _i-\lambda _1\right) \omega^1 _{\text{ii}} \lambda _j-\left(\lambda _1-\lambda _j\right){}^2 \omega^1_{\text{jj}}.
\end{align*}

  Using $\lambda_1=-\frac{n\, H}{2}$ in  (\ref{e:d12}), we obtain
  \begin{eqnarray}\label{e:d344}
 -2\lambda_i \, e_1(\lambda_1) = (\lambda_{i}-\lambda_{1})^2 \, \omega_{ii}^{1},\qquad
 \mbox{for some}\ \  i\in B.
\end{eqnarray}

  Eliminating $e_1(\lambda_1)$ from (\ref{e:d343}) using (\ref{e:d344}), we find
   \begin{eqnarray}\label{e:d345}
  b_6-b_7 \,\omega^1 _{\text{jj}}+6 \lambda _i^2 \left(\lambda _1-\lambda _j\right)^2 (\omega^1 _{\text{jj}})^2=0,
    \end{eqnarray}
     where
    \begin{align*}
   b_6&= 2 \Big(\lambda _1^2 \lambda _i \left(\lambda _i \left(c-2 \lambda _j^2\right)-\lambda _j \left(c+6 (\omega^1 _{\text{ii}})^2\right)+2 \lambda _i^2 \lambda _j\right)+\lambda _i^2 \lambda _j \left(c \lambda _j-\lambda _i \left(c+4 (\omega^1 _{\text{ii}})^2\right)\right)\\&+\lambda _1 \lambda _i^2 \lambda _j \left(-\lambda _i^2+9 (\omega^1 _{\text{ii}})^2+\lambda _j^2\right)+\lambda _1^3 (\omega^1 _{\text{ii}})^2 \lambda _j\Big),\\
   b_7&=2 \left(\lambda _1-\lambda _i\right) \lambda _i \omega^1 _{\text{ii}} \left(\lambda _1-\lambda _j\right){}^2.
   \end{align*}

   Eliminating $\omega^1_{jj}$ from (\ref{e:d345}) using (\ref{e:d342}), we have
   \begin{eqnarray}\label{e:d346}
 3 \left(\lambda _1-\lambda _i\right)^2 (\omega^1_{ii})^2 \lambda _j \left(\lambda _1^2-\lambda _i \lambda _j\right)-\lambda _i \left(\lambda _1-\lambda _j\right)^2 \\ \nonumber \left(\lambda _1^2 \left(c+2 \lambda _i \lambda _j\right)-c \lambda _i \lambda _j-\lambda _1 \lambda _i \lambda _j \left(\lambda _i+\lambda _j\right)\right)=0.
  \end{eqnarray}

Differentiating  (\ref{e:d346}) with respect to $e_1$ and using (\ref{e:d32}),  (\ref{e:c14}) and (\ref{e:d344}), we obtain
 \begin{eqnarray}\label{e:d347}
 \omega ^1_{\text{jj}} \left(a_4 (\omega^1 _{ii})^2+a_3\right)+a_2 (\omega^1 _{ii})^3+a_1 \omega ^1_{ii}=0,
   \end{eqnarray} 
   where
   \begin{align*}
   a_1 &= -\left(\lambda_1 - \lambda_i\right) \lambda_i 
   \Bigg(\lambda_1^4 \left(6c + 28 \lambda_i \lambda_j\right) 
   - \lambda_1^3 \big(\lambda_i \left(4c + 35 \lambda_j^2\right) 
   - 2c \lambda_j + 29 \lambda_i^2 \lambda_j\big) + \lambda_1^2 \lambda_j \big(-12 \lambda_i (c - 2 \lambda_j^2) 
   \\
   &\quad + 4c \lambda_j + 19 \lambda_i^2 \lambda_j + 3 \lambda_i^3\big)+ \lambda_1 \lambda_i \lambda_j \big(\lambda_i \left(2c - 15 \lambda_j^2\right) 
   - 4c \lambda_j + 8 \lambda_i^2 \lambda_j - 5 \lambda_j^3\big) \\
   &\quad + \lambda_i \lambda_j^2 \big(\lambda_i \left(10c + \lambda_j^2\right) 
   - 4c \lambda_j + \lambda_i^2 \lambda_j\big) \Bigg), \\
   a_2 &= 6 \left(\lambda_1 - \lambda_i\right)^2 \lambda_j 
   \big(\lambda_1 \lambda_i \left(\lambda_i - 2 \lambda_j\right) 
   + 6 \lambda_i^2 \lambda_j - 7 \lambda_1^2 \lambda_i + 2 \lambda_1^3\big), \\
   a_3 &= 2 \lambda_i^2 \left(\lambda_1 - \lambda_j\right)^2 
   \Big(\lambda_1^2 \left(2c + 8 \lambda_i \lambda_j + \lambda_i^2\right) 
   + \lambda_1 \lambda_i \left(c - 3 \lambda_i \lambda_j - 4 \lambda_j^2\right) - 3c \lambda_i \lambda_j - 2 \lambda_1^3 \lambda_i\Big), \\
   a_4 &= 6 \left(\lambda_1 - \lambda_i\right)^2 \lambda_i \left(\lambda_1 - \lambda_j\right) 
   \left(\lambda_1^2 - 2 \lambda_i \lambda_j\right).
\end{align*}

Eliminating $\omega^1_{jj}$ from (\ref{e:d347}) using (\ref{e:d342}), we have
 \begin{eqnarray}\label{e:d348}
a_5- (\omega ^1_{ii})^2\,a_6 =0,
   \end{eqnarray} 
   where
  \begin{align*}
 a_5 &= \lambda_i \left(\lambda_1 - \lambda_j\right) \Bigg(
        \lambda_1^4 \left(6c + 32 \lambda_i \lambda_j\right) 
        - \lambda_1^3 \Big(\lambda_i \left(4c + 51 \lambda_j^2\right) 
        + 2c \lambda_j + 35 \lambda_i^2 \lambda_j \Big)+ \lambda_1^2 \lambda_j \Big(\lambda_i \left(32 \lambda_j^2 - 10c\right) \\
        &\quad 
        + 4c \lambda_j + 41 \lambda_i^2 \lambda_j + 5 \lambda_i^3 \Big) + \lambda_1 \lambda_i \lambda_j \Big(\lambda_i \left(4c - 23 \lambda_j^2\right) 
        + 2c \lambda_j + 2 \lambda_i^2 \lambda_j - 5 \lambda_j^3 \Big) \\
        &\quad + \lambda_i \lambda_j^2 \Big(\lambda_i \left(4c + \lambda_j^2\right) 
        - 4c \lambda_j + \lambda_i^2 \lambda_j \Big) \Bigg),\\
   a_6&= 6 \left(\lambda _1-\lambda _i\right) \lambda _j \Big(-\lambda _1^3 \left(9 \lambda _i+2 \lambda _j\right)+\lambda _1^2 \lambda _i \left(2 \lambda _i+3 \lambda _j\right) +\lambda _1 \lambda _i \lambda _j \left(9 \lambda _i+2 \lambda _j\right)\\&-2 \lambda _i^2 \lambda _j \left(\lambda _i+3 \lambda _j\right)+3 \lambda _1^4\Big).
\end{align*}

Eliminating $\omega^1_{ii}$ from (\ref{e:d348}) using (\ref{e:d346}), we find
 \begin{eqnarray}\label{e:d349}
 \lambda _i \left(\lambda _1-\lambda _j\right)\Big(  -8 c \lambda _1^4 \lambda _i+a_{10} \,\lambda _j+a_{9} \, \lambda _j^2+a_8 \,\lambda _j^3+a_7 \,\lambda _j^4\Big)=0,
   \end{eqnarray}
where
 \begin{align*}
 a_7 &= \lambda _i^3+6 \lambda _1 \lambda _i^2+\lambda _1^2 \lambda _i,\\
   a_8&= 8 c \lambda _i^2+\lambda _i^4-8 \lambda _1 \lambda _i^3-3 \lambda _1^2 \lambda _i^2-14 \lambda _1^3 \lambda _i,\\
   a_9&= -32 c \lambda _1 \lambda _i^2+8 c \lambda _i^3+25 \lambda _1^4 \lambda _i-4 \lambda _1^3 \lambda _i^2-3 \lambda _1^2 \lambda _i^3+6 \lambda _1 \lambda _i^4,\\
   a_{10}&=32 c \lambda _1^3 \lambda _i-8 c \lambda _1^4-20 \lambda _1^5 \lambda _i+25 \lambda _1^4 \lambda _i^2-14 \lambda _1^3 \lambda _i^3+\lambda _1^2 \lambda _i^4.
\end{align*}

From (\ref{e:d349}), we get
\begin{eqnarray}\label{e:d350}
 -8 c \lambda _1^4 \lambda _i+a_{10} \,\lambda _j+a_{9} \, \lambda _j^2+a_8 \,\lambda _j^3+a_7 \,\lambda _j^4=0.
\end{eqnarray}

Also, from (\ref{e:d37}), we have
\begin{eqnarray}\label{e:d351}
 \lambda _i d_i-(\lambda _i+\lambda _1)^2=0,
\end{eqnarray}
and
\begin{eqnarray}\label{e:d352}
 \lambda _j  d_j-\left(\lambda _j+\lambda _1\right){}^2=0,
\end{eqnarray}
where $d_i=\pm c_i^2$ and $d_j=\pm c_j^2.$

We write
equations (\ref{e:d350}) and (\ref{e:d352}) as polynomials
${\mathcal{F}_1}_{(\lambda_1,\lambda_i)}(\lambda_j),
{\mathcal{G}_1}_{(\lambda_1,\lambda_i)}(\lambda_j)$ with
coefficients in polynomial ring $\mathbb{R}[\lambda_1,\lambda_i]$ over real
field. According to Lemma 2.1, the equations
${\mathcal{F}_1}_{(\lambda_1,\lambda_i)}(\lambda_j)=0,
{\mathcal{G}_1}_{(\lambda_1,\lambda_i)}(\lambda_j)=0$ have a
common root if and only if
\begin{equation}\label{e:n31}\begin{array}{lcl}
\mathfrak{R}({\mathcal{F}_{1}}_{(\lambda_1,\lambda_i))}(\lambda_j),
{\mathcal{G}_1}_{(\lambda_1,\lambda_i)}(\lambda_j))=0,
\end{array}\end{equation}
which is a polynomial equation
\begin{equation}\label{e:d61}\begin{array}{rcl}
{\mathcal{F}_{2}}(\lambda_i, \lambda_1)=0,
\end{array}
 \end{equation}
where
${\mathcal{F}_{2}}(\lambda_i, \lambda_1)= 64 c^2 \lambda _1^{10}+a_{11} \lambda _i^8+a_{12} \lambda _i^7+a_{13} \lambda _i^6+a_{14} \lambda _i^5+a_{15} \lambda _i^4+a_{16} \lambda _i^3+a_{17} \lambda _i^2+a_{18} \lambda _i$ and
\begin{align*}
a_{11}&=16 \lambda _1^6+8 d_j \lambda _1^5+d_j^2 \lambda _1^4,\\
a_{12}&=160 \lambda _1^7-64 d_j \lambda _1^6-22 d_j^2 \lambda _1^5+64 c \lambda _1^5+d_j^3 \lambda _1^4+16 c d_j \lambda _1^4,\\
a_{13}&=240 \lambda _1^8-552 d_j \lambda _1^7+215 d_j^2 \lambda _1^6-8 d_j^3 \lambda _1^5-256 c d_j \lambda _1^5+64 c^2 \lambda _1^4+16 c d_j^2 \lambda _1^4,\\
a_{14}&=-320 \lambda _1^9+128 d_j \lambda _1^8+140 d_j^2 \lambda _1^7-2240 c \lambda _1^7-58 d_j^3 \lambda _1^6+1424 c d_j \lambda _1^6-640 c^2 \lambda _1^5\\&-64 c d_j^2 \lambda _1^5-8 c d_j^3 \lambda _1^4+64 c^2 d_j \lambda _1^4,\\
a_{15}&=2800 \lambda _1^{10}-648 d_j \lambda _1^9-417 d_j^2 \lambda _1^8+1024 c \lambda _1^8+116 d_j^3 \lambda _1^7+896 c d_j \lambda _1^7+960 c^2 \lambda _1^6-960 c d_j^2 \lambda _1^6\\&+160 c d_j^3 \lambda _1^5+256 c^2 d_j \lambda _1^5-8 c d_j^4 \lambda _1^4-64 c^2 d_j^2 \lambda _1^4,\\
a_{16}&=-2400 \lambda _1^{11}+320 d_j \lambda _1^{10}+426 d_j^2 \lambda _1^9-2880 c \lambda _1^9-95 d_j^3 \lambda _1^8+4400 c d_j \lambda _1^8+3328 c^2 \lambda _1^7-2400 c d_j^2 \lambda _1^7\\&+592 c d_j^3 \lambda _1^6-3712 c^2 d_j \lambda _1^6-48 c d_j^4 \lambda _1^5+896 c^2 d_j^2 \lambda _1^5-64 c^2 d_j^3 \lambda _1^4,\\
a_{17}&=3600 \lambda _1^{12}-2264 d_j \lambda _1^{11}+425 d_j^2 \lambda _1^{10}-5120 c \lambda _1^{10}-20 d_j^3 \lambda _1^9+4480 c d_j \lambda _1^9+960 c^2 \lambda _1^8\\&-1360 c d_j^2 \lambda _1^8+160 c d_j^3 \lambda _1^7+256 c^2 d_j \lambda _1^7-8 c d_j^4 \lambda _1^6-64 c^2 d_j^2 \lambda _1^6,\\
a_{18}&=960 c \lambda _1^{11}-720 c d_j \lambda _1^{10}-640 c^2 \lambda _1^9+160 c d_j^2 \lambda _1^9-8 c d_j^3 \lambda _1^8+64 c^2 d_j \lambda _1^8.
\end{align*}

Rewrite,
equations (\ref{e:d61}) and (\ref{e:d351}) respectively as polynomials
${\mathcal{F}_2}_{(\lambda_1)}(\lambda_i),
{\mathcal{G}_2}_{(\lambda_1)}(\lambda_i)$ with
coefficients in polynomial ring $\mathbb{R}[\lambda_1]$ over real
field. According to Lemma 2.1, the equations
${\mathcal{F}_2}_{(\lambda_1)}(\lambda_i)=0,
{\mathcal{G}_2}_{(\lambda_1)}(\lambda_i)=0$ have a
common root if and only if
\begin{equation}\label{e:n31}\begin{array}{lcl}
\mathfrak{R}({\mathcal{F}_{2}}_{(\lambda_1))}(\lambda_i),
{\mathcal{G}_1}_{(\lambda_1)}(\lambda_i))=0,
\end{array}\end{equation}
which results in the polynomial equation
\begin{equation}\label{e:d62}\begin{array}{rcl}
{\mathcal{F}_{3}}(\lambda_1)=\sum_{k=0}^{14} l_k \lambda_1^k = 0, 
\end{array}
 \end{equation}
where $ {\mathcal{F}_{3}}(\lambda_1)$ is a polynomial of degree $14$ in $\lambda_1$ with real constant coefficients $l_k$. The leading coefficient is $l_{14} = 84934656$, while $l_{13} = -70778880 d_i - 70778880 d_j$, and so on. Since $\lambda_1$ must satisfy this polynomial equation, it follows that $\lambda_1$ is a real constant, leading to a contradiction.

 Thus the proof is complete.
  \hfill$\Box$

\section{\textbf{Concluding Remarks}}

\noindent
{\bf 1)}
The converses of Corollary \ref{cor1} and Theorem \ref{main2} are not true in general.

Indeed, consider the classical $3$-dimensional  catenoid in $M^3\subset \mathbb R^4$
parametrized by
\[
X(s,\theta,\phi) = \Big(r(s)\sin\theta\cos\phi,\, r(s)\sin\theta\sin\phi,\, r(s)\cos\theta,\, w(s)\Big),\]
where
\[ 0<\theta<\pi, \qquad 0<\phi<2\pi,\]
generated by rotating the profile curve
\[ \gamma(s)=(r(s),w(s)) \]
in the $(r,x_4)$-plane satisfying 
$r'(s)^2+w'(s)^2=1$.
The induced metric is
\[
g = ds^2 + r(s)^2 g_{\mathbb{S}^2}.
\]
The principal curvatures are
\[
\kappa_0 = -\frac{r''}{\sqrt{1-r'^2}}, \qquad \kappa_1 = \kappa_2 = \frac{\sqrt{1-r'^2}}{r}.
\]
The minimal condition \(H = 0\) yields the ODE
\[
r'' = \frac{2(1-r'^2)}{r}.
\]
The first integral is
\[
r'^2 = 1 - \frac{a^4}{r^4}, \quad a > 0,
\]
and
\[
r'' = \frac{2a^4}{r^5}.
\]

In the orthonormal frame \(e_0 = \partial_s\), \(e_1 = r^{-1}\partial_\theta\), \(e_2 = (r\sin\theta)^{-1}\partial_\phi\), the Ricci tensor $S$ is diagonal with eigenvalues
\[
\lambda_0 = -\frac{4a^4}{r^6}, \qquad \lambda_1 = \lambda_2 = -\frac{a^4}{r^6}.
\]

The relevant connection coefficients are
\[
\nabla_{e_0} e_1 = 0, \quad \nabla_{e_1} e_0 = \frac{r'}{r} e_1, \quad \nabla_{e_1} e_1 = -\frac{r'}{r} e_0.
\]

Direct computation gives
\[
(\nabla_{e_0} S)(e_1,e_1) = \lambda_1' = \frac{6a^4 r'}{r^7},
\]
\[
(\nabla_{e_1} S)(e_0,e_1) = \frac{r'}{r}(\lambda_0 - \lambda_1) = -\frac{3a^4 r'}{r^7}.
\]

Thus,
\[
(\nabla_{e_0} S)(e_1,e_1) - (\nabla_{e_1} S)(e_0,e_1) = \frac{9a^4 r'}{r^7} \neq 0
\]
(since \(r' \not\equiv 0\) away from the neck). Therefore, $S$ is \textbf{not a Codazzi tensor}, so the hypersurface does \textbf{not} have harmonic curvature.

\medskip
Furthermore, \(\lambda_0' = \frac{24a^4 r'}{r^7} \neq 0\), so \(\nabla S \neq 0\) (the Ricci tensor is not parallel).

This minimal (\(H=0\)) hypersurface $M^3$ is therefore a CMC counterexample showing that neither parallel Ricci tensor nor harmonic curvature follows from the CMC condition alone.

We also remark that equations (4.12) and (4.13) show that constant mean curvature alone does not imply that Ricci tensor is parallel or the condition of harmonic curvature. The $3$-dimensional catenoid in $\mathbb{R}^{4}$ is a counterexample, as we just saw. However, if one additionally assumes that the shape operator is parallel 
($\nabla \bold{A}=0$), equivalently that the second fundamental form is parallel, then $\nabla S=0$. Hence the condition that the shape operator is parallel is a natural sufficient condition under which the converse statement holds.

\medskip
\noindent
{\bf 2)} 
Our Theorem \ref{main2}  shows that every biconservative hypersurface in a space form with harmonic curvature, is necessarily of constant mean curvature.
Therefore, we can directly apply the classification theorem of  \'{O}machi (\cite{h4})  to obtain the following:

\begin{cor}[Compact Case]\label{cor:compact}
Let \(M^{n-1}\) be a compact biconservative hypersurface with harmonic curvature immersed in a space form.
\begin{enumerate}
    \item If the ambient space is \(\mathbb{R}^n\), then \(M^{n-1}\) is a standard hypersphere \(\mathbb{S}^{n-1}\).
    \item If the ambient space is \(\mathbb{S}^n\), then \(M^{n-1}\) is a Clifford torus \(\mathbb{S}^p(r) \times \mathbb{S}^{n-1-p}(s)\) for some \(0 \leq p \leq n-1\).
\end{enumerate}
\end{cor}

\medskip
\noindent
{\bf 3)}
The local theory developed by Umehara (\cite[Theorem 2.1, p. 83]{h1})  shows that a CMC hypersurface with harmonic curvature in a space form is locally isoparametric and possesses at most two distinct principal curvatures.
Combining this with our theorem, every biconservative hypersurface with harmonic curvature is locally one of the classical isoparametric models.

Hence any noncompact example is locally an open subset of a standard isoparametric hypersurface (such as a generalized cylinder in Euclidean space or a Clifford hypersurface in the sphere).
We emphasize that this is a local geometric classification consequence of Umehara's theorem.

\medskip
\noindent
The authors have no competing interests to declare that are relevant to the content of this article.



\bigskip
\noindent
Author's address:\\
\textbf{Ram Shankar Gupta}\\
University School of Basic and Applied
Sciences, Guru Gobind Singh Indraprastha University, Sector-16C,
Dwarka, New Delhi-110078,
India\\
\textbf{Email:} ramshankar.gupta@gmail.com

\smallskip
\noindent
\textbf{Andreas Arvanitoyeorgos}\\
University of Patras, Department of Mathematics, GR-26504 Rion, Greece, and\\
Hellenic Open University, Aristotelous 18, GR-26335 Patras, Greece\\
\textbf{Email:} arvanito@math.upatras.gr

\smallskip
\noindent
\textbf{Marina Statha}\\
University of Patras, Department of Mathematics, GR-26504 Rion, Greece\\
\textbf{Email:} marina.statha@gmail.com
 \end{document}